\documentclass[a4paper,11 pt,reqno]{amsart}
\usepackage{a4,amsbsy,amscd,amssymb,amstext,amsmath,latexsym,oldgerm,enumerate,verbatim,graphicx,xy}

\makeatletter                                       
\def\l@subsection{\@tocline{2}{0pt}{2.5pc}{2.5pc}{}}
\makeatother                                        

\makeatletter                                                  
\def\chapter{\clearpage\thispagestyle{plain}\global\@topnum\z@ 
\@afterindenttrue \secdef\@chapter\@schapter}
\makeatother

\newtheorem{thmgl} {Theorem}    
\newtheorem{propgl}{Proposition}
\newtheorem{lemgl} {Lemma}
\newtheorem{corgl} {Corollary}

\newtheorem{cornn}{Corollary}

\theoremstyle{definition}

\newtheorem{remgl} {Remark}
\newtheorem{remsgl} [remgl]{Remarks}

\newcommand{\mf}{\mathfrak}

\newcommand{\mb}{\mathbb}

\newcommand{\nts}{\negthinspace}     

\newcommand{\ov}{\overline}

\newcommand{\sm}{\setminus}         

\newcommand{\ot}{\otimes}           

\newcommand{\id}{{\rm id}}
\newcommand{\gr}{{\rm gr}}   

\newcommand{\g}{\mf{g}}

\let\ttie\t
\newcommand{\tie}[1]{{\let\t\ttie \ttie#1}}
\renewcommand{\t}{\mf{t}}  
\newcommand{\n}{\mf{n}}

\newcommand{\spl}{\mf{sl}}

\newcommand{\pgl}{\mf{pgl}}

\newcommand{\Lie}{{\rm Lie}}
\newcommand{\Dist}{{\rm Dist}}
\newcommand{\GL}{{\rm GL}}
\newcommand{\PGL}{{\rm PGL}}

\newcommand{\SO}{{\rm SO}}

\newcommand{\e}{\epsilon}
\newcommand{\ve}{\varepsilon}

\xyoption{all}

\begin{document}

\title{Factorisation properties of group scheme actions}

\begin{abstract}
Let $H$ be an algebraic group scheme over a field $k$ acting on a commutative $k$-algebra $A$ which is a unique factorisation domain. We show that, under certain mild assumptions, the monoid of nonzero $H$-stable principal ideals in $A$ is free commutative. From this we deduce, in certain special cases, results about the monoid of nonzero semi-invariants and the algebra of invariants. We use an infinitesimal method which allows us to work over an arbitrary base field.
\end{abstract}

\author[R.\ H.\ Tange]{Rudolf Tange}

\keywords{}
\thanks{2010 {\it Mathematics Subject Classification}. 13F15, 14L15, 14L30}

\maketitle

\section{Introduction and notation}\label{s.intro}

Let $H$ be a linear algebraic group acting on an irreducible variety $X$. It is of interest to know conditions for when a function $f$ on $X$ is an $H$-(semi-)invariant. For example, when $X$ is normal, $H$ is connected and the principal divisor $(f)$ is fixed by a closed subgroup $H'$ of $H$, then $f$ is an $H'$-semi-invariant. 
Without the action of the bigger connected group $H$ this is no longer true.
Another well-known fact is that when $H$ is connected, $X$ is affine and $k[X]$ is a unique factorisation domain (UFD) of which the invertible elements are the nonzero constants,
then the prime factors of an $H$-semi-invariant are $H$-semi-invariants. For non-connected groups and/or invariants the situation is much more complicated. A standard example where the ring of invariants in a UFD is not a UFD is $k[\PGL_n]=k[\GL_n]^Z$, where $Z$ consists of the nonzero multiples of the identity acting via the right regular representation. See \cite[Ch.~3]{Ben} for the case of invariants for finite groups.

The quotient space $X/H$ (when it exists) is often also involved here. By \cite[III \S3 no. 2,5]{DG} every homogeneous space for a linear algebraic group $G$ is of the form $G/H$, where $H$ is a closed subgroup scheme of $G$ (so $k[H]$ need not be reduced). See \cite[V.17]{Bo} for a special case. So even if one is only interested in homogeneous spaces (over fields of positive characteristic) one is led to consider group scheme actions.

In this paper we study the behaviour of factorisation with respect to the action of a group scheme $H$. Our initial interest was in (semi-)invariants, but it turns out, for reasons partly indicated above, that we first have to look at the property ``$Aa$ is $H$-stable" of an element $a\in A$. After that we indicate ways to go from``$Aa$ is $H$-stable" to ``$a$ is an $H$-semi-invariant" (see Remark~\ref{rems.UFD}.1 and Proposition~\ref{prop.semi-invariant}). The main results are Theorems~\ref{thm.UFD1} and \ref{thm.UFD2}. They state roughly that when a group scheme $H$ acts on a UFD $A$, then the monoid of nonzero elements $a\in A$ such that $Aa$ is $H$-stable has the unique factorisation property.

Throughout this paper $k$ denotes a field. All group schemes are affine and $H$ will always denote an algebraic group scheme over $k$. For the basic definitions concerning (group) schemes we refer to \cite{Jan}. By ``algebraic group" we will always mean a reduced linear algebraic group over an algebraically closed field. By ``reductive group" we will always mean an algebraic group with trivial unipotent radical. An element of a module for a group scheme is called a {\it semi-invariant} if it spans a submodule.

\section{Group scheme actions on UFD's}
We start with a basic result about distributions. We need a simple lemma which is easily proved using the modular property for subspaces of a vector space and induction on $r$.
\begin{lemgl}\label{lem.subspaces}
Let $V_0\supseteq V_1\supseteq V_2\supseteq\cdots\supseteq V_r$ and $W_0\supseteq W_1\supseteq W_2\supseteq\cdots\supseteq W_r$ be two descending chains of subspaces of a vector space $V$. Then $$\bigcap_{i=0}^r(V_i+W_{r-i})=V_r+W_r+\sum_{i=0}^{r-1}V_i\cap W_{r-1-i}.$$
\end{lemgl}
Now let $X,X'$ be affine algebraic $k$-schemes and let $x\in X(k)$ and $x'\in X'(k)$. As in \cite{Jan} we denote the space of distributions of $X$ with support in $x$ by $\Dist(X,x)\subseteq k[X]^*$. It consists of the functionals that vanish on some power of the vanishing ideal $I_x$ of $x$. This space has a natural filtration $(\Dist_n(X,x))_{n\ge0}$, where $\Dist_n(X,x)$ consists of the functionals that vanish on the $(n+1)$-st power of the vanishing ideal $I_x$ of $x$. If we apply Lemma~\ref{lem.subspaces} with $r=n$, $V_i=I_x^{i+1}\ot k[X']$ and $W_i=k[X]\ot I_{x'}^{i+1}$, then we obtain, as in \cite[I.7.4]{Jan},

\begin{equation}\label{eq.dist1}
\,\nts\bigcap_{i=0}^n(I_x^{i+1}\ot k[X']+k[X]\ot I_{x'}^{n+1-i})=
I_x^{n+1}\ot k[X']+k[X]\ot I_{x'}^{n+1}+\nts\sum_{i=1}^{n}I_x^{i}\ot I_{x'}^{n+1-i}\nts
\end{equation}
which is clearly equal to $\sum_{i=0}^{n+1}I_x^i\ot I_{x'}^{n+1-i}$. From this it is deduced in \cite[I.7.4]{Jan} that there is an isomorphism $\Dist(X,x)\ot\Dist(X',x')\cong\Dist(X\times X',(x,x'))$ which maps each $\sum_{i=0}^n\Dist_i(X,x)\ot\Dist_{n-i}(X,x)$ onto $\Dist_n(X\times X',(x,x'))$. Taking $X=X'$ and $x=x'$ we see that $\Dist(X,x)$ is a coalgebra. Its comultiplication $\Delta$ is the ``differential" of the diagonal embedding at $(x,x)$. The counit is the evaluation at the constant function $1\in k[G]$.

If we apply Lemma~\ref{lem.subspaces} with $r=n-2$, $V_i=I_x^{i+2}\ot k[X']$ and $W_i=k[X]\ot I_{x'}^{i+2}$, then we obtain
\begin{equation}\label{eq.dist2}
\bigcap_{i=1}^{n-1}(I_x^{i+1}\ot k[X']+k[X]\ot I_{x'}^{n+1-i})=I_x^n\ot k[X']+k[X]\ot I_{x'}^n+\sum_{i=2}^{n-1}I_x^{i}\ot I_{x'}^{n+1-i}.
\end{equation}
\begin{lemgl}\label{lem.distribution}
Let $X$ be an algebraic affine $k$-scheme, let $x\in X(k)$, let $n>0$ and denote the evaluation at $x$ by $\ve_x$. Then we have for all $u\in\Dist_n(X,x)$ that
$$\Delta(u)-u\ot\ve_x-\ve_x\ot u\in\sum_{i=1}^{n-1}\Dist_i(X,x)\ot\Dist_{n-i}(X,x).$$
\end{lemgl}

\begin{proof}
By \eqref{eq.dist2} $\sum_{i=1}^{n-1}\Dist_i(X,x)\ot\Dist_{n-i}(X,x)$ is everything that vanishes on $$I_x^n\ot k[X]+k[X]\ot I_{x}^n+\sum_{i=2}^{n-1}I_x^i\ot I_{x}^{n+1-i}.$$
From \eqref{eq.dist1} and the fact that $\Delta$ is filtration preserving we get $\Delta(u)$ vanishes on $\sum_{i=2}^{n-1}I_x^i\ot I_{x}^{n+1-i}$ and, clearly, the same holds for $u\ot\ve_x$ and $\ve_x\ot u$, so we only have to check that $\Delta(u)-u\ot\ve_x-\ve_x\ot u$ vanishes on $I_x^n\ot k[X]$ and $k[X]\ot I_x^n$. This follows immediately from $u(I_x^{n+1})=0$, $\Delta(u)(f\ot g)=u(fg)$ and $g=g-g(x)+g(x)$ for $f,g\in k[X]$.
\end{proof}
A simple induction gives us the following generalisation.
\begin{cornn}
Let $X,x,n,\ve_x$ be as in Lemma~\ref{lem.distribution}, let $m\ge2$ and let $\Delta_m: k[X]\to k[X]^{\ot m}$ be the $m$-th comultiplication. Then we have for all $u\in\Dist_n(X,x)$ that
\begin{align*}
\Delta_m(u)-u\ot\ve_x\ot\cdots\ot\ve_x-\cdots-\ve_x\ot\cdots\ot\ve_x\ot u\\
\in\sum_{i_1,\ldots,i_m}\Dist_{i_1}(X,x)\ot\cdots\ot\Dist_{i_m}(X,x),
\end{align*}
where the sum is over all $i_1,\ldots,i_m$ in $\{0,\ldots,n-1\}$ with $\sum_{j=1}^mi_j=n$.
\end{cornn}

We will apply the above lemma and its corollary to the case that $X=H$, where $H$ is an (affine) algebraic group scheme and $x$ is the identity $e\in H$. Then we write $\Dist(H,e)=\Dist(H)$. Since $H$ is a group scheme, $\Dist(H)$ is not just a coalgebra, but a Hopf algebra. Its unit element it the evaluation $\ve$ at $e$. We note that $\Dist(H)=\Dist(H^0)$, where $H^0$ is the identity component of $H$. So we can speak about $\Dist(H)$-modules (because of the algebra structure) and about the tensor product of $\Dist(H)$-modules. Recall that $H$ is called {\it finite} when $\dim k[H]<\infty$ and {\it infinitesimal} when $k[H]$ is finite dimensional and has a unique maximal ideal. Every $H$-module is a $\Dist(H)$-module and if $H$ is infinitesimal, then this gives an equivalence of categories. In general a (left) $H$-module (\cite[I.2.7]{Jan}) is the same thing as a right $k[H]$-comodule (for $k[H]$ as a coalgebra). If $k$ is perfect of characteristic $p>0$, $r\ge0$ and $H$ is an algebraic group scheme and over $k$, then the {\it $r$-th Frobenius kernel $H_r$ of $H$} (see \cite[I.9]{Jan}) is an infinitesimal group scheme and we have $\Dist(H)=\bigcup_{r\ge0}\Dist(H_r)$.

Let $A$ be a $k$-algebra. We say that $A$ is a {\it $\Dist(H)$-algebra}, if it is a $\Dist(H)$-module such that the multiplication $A\ot A\to A$ is a morphism of $\Dist(H)$-modules. We say that $A$ is an {\it $H$-algebra} or that $H$ acts on $A$, if $A$ is an $H$-module such that, for each commutative $k$-algebra $R$, $H(R)$ acts by automorphisms of the algebra $R\ot A$. This condition is equivalent to the condition that the comodule map $\Delta_A:A\to A\ot k[H]$ is a homomorphism of algebras. Every $H$-algebra is a $\Dist(H)$-algebra and when $H$ is infinitesimal, the two notions coincide.

Let $A$ be a commutative $\Dist(H)$-algebra. Assume first that $k$ is perfect with ${\rm char}\, k=p>0$ and that $H$ is infinitesimal of height $\le r$ (i.e. $H=H_r$). Then we have for $f\in k[H]$ that $f^{p^r}=(f-f(e))^{p^r}+f(e)^{p^r}=f(e)^{p^r}$. Furthermore, if $\Delta_A(a)=\sum_ia_i\ot f_i$, then $a=\sum_if_i(e)a_i$. Finally, $A$ is an $H$-algebra, since $H$ is infinitesimal. So the comodule map is a homomorphism of algebras. From these facts we easily deduce that $H$ fixes the elements of $A^{p^r}$. Now assume only that ${\rm char}\, k=p>0$. Then we deduce from the above, by base extension to a perfect field, that $u\in\Dist(H)$ acts $A^{p^r}$-linearly on $A$ whenever it kills the $p^r$-th powers of the elements of $I_e$.

Now drop the assumption on $k$ and assume that $A$ is a domain. Then the $\Dist(H)$-action extends to give ${\rm Frac}(A)$ the structure of a $\Dist(H)$-algebra (clearly such an extension is unique). To see this in case ${\rm char}\,k=0$ one can use Cartier's Theorem (\cite[II.6.1.1]{DG}) which says that $\Dist(H)=U(\Lie(H))$. Then the action of $\Dist(H)$ is given by the standard extension of derivations to the field of fractions. In case ${\rm char}\,k=p>0$ we use that, by the above, for $n\le p^r-1$, the elements of $\Dist_n(H)$ act $A^{p^r}$-linearly on $A$. So one can extend their action to ${\rm Frac}(A)=A[(A^{p^r}\sm\{0\})^{-1}]$ in the obvious way. This clearly leads to the required action of $\Dist(H)$. If $A$ is an $H$-algebra, then there is, of course, also an action of $H(k)$ on ${\rm Frac}(A)$. It is important to note that the actions of $\Dist(H)$ and $H(k)$ on ${\rm Frac}(A)$ are in general not locally finite.

If $A$ is a $\Dist(H)$-algebra and $a\in A$, then $Aa$ is $\Dist(H)$-stable if and only if $u\cdot a\in Aa$ for all $u\in\Dist(H)$. If, in addition, $A$ is a commutative domain, then we also have for $a\in{\rm Frac}(A)$ that $Aa$ is $\Dist(H)$-stable if and only if $u\cdot a\in Aa$ for all $u\in\Dist(H)$. We note that if a group $G$ acts on a commutative domain $A$ by automorphisms, then $Aa$ is $G$-stable if and only if, for all $g\in G$, $a$ and $g(a)$ differ by a unit.

The first assertion of the lemma below is \cite[I.7.17(6), 8.6]{Jan}. To prove the second assertion one may assume that $k_1=k$. Then one takes a $k$-point from each irreducible component of $H_k$ and the result follows from \cite[I.7.17(6), 8.6]{Jan} and some elementary properties of the comodule map. We leave the details to the reader.
\begin{lemgl}\label{lem.submodule}
Let $H$ be an algebraic group scheme over $k$ and let $M$ be an $H$-module. Assume that $k$ is perfect or that $H$ is finite and let $N$ be a subspace of $M$. Denote the identity component of $H$ by $H^0$.
\begin{enumerate}[{\rm (i)}]
\item $N$ is an $H^0$-submodule if and only if $N$ is $\Dist(H)$-stable.
\item If $k_1$ is an extension field of $k$ such that every irreducible component of $H_{k_1}$ contains a $k_1$-point, then $N$ is an $H$-submodule if and only if $N$ is $\Dist(H)$-stable and $k_1\ot N$ is $H(k_1)$-stable.
\end{enumerate}
\end{lemgl}

\begin{propgl}\label{prop.equifactorisation}
Let $H$ be an algebraic group scheme over $k$ and let $A$ be a $\Dist(H)$-algebra which is a commutative domain. Let $a,b,c\in A$ and $m\ge0$.
\begin{enumerate}[{\rm(i)}]
\item Assume that $A$ is a UFD and that $b$ and $c$ are coprime. If $Abc$ is $\Dist(H)$-stable, then so are $Ab$ and $Ac$.
\item Assume that $A$ is a UFD and that $b$ and $c$ are coprime. If $A(b/c)$ is $\Dist(H)$-stable, then so are $Ab$ and $Ac$.
\item If $Aa^m$ is a $\Dist(H)$-stable and $m\ne0$ in $k$, then $Aa$ is $\Dist(H)$-stable.
\end{enumerate}
\end{propgl}

\begin{proof}
(i).\ We show by induction on $n$ that $Ab$ is $\Dist_n(H)$-stable. For $n=0$ there is nothing to prove. So assume that $n>0$. Let $u\in\Dist_n(H)$. By Lemma~\ref{lem.distribution} we can write $\Delta(u)-u\ot\ve-\ve\ot u=\sum_{i=1}^{n-1}u^1_i\ot u^2_{n-i}$, where $\ve$ is the evaluation at the unit element of $H$ and $u_i^j\in\Dist_i(H)$. Since $A$ is a $\Dist(H)$-algebra we have $u\cdot(bc)-(u\cdot b)c-b(u\cdot c)=\sum_{i=1}^{n-1}(u^1_i\cdot b)(u^2_{n-i}\cdot c)$. So by our assumption and the induction hypothesis we get that $b$ divides $(u\cdot b)c$. Since $b$ and $c$ are coprime, this means that $b$ divides $(u\cdot b)$.\\
(ii).\ Let $u\in\Dist_n(H)$. Since $b=\frac{b}{c}\,c$, we have by Lemma~\ref{lem.distribution}
$$u\cdot b-(u\cdot\frac{b}{c})\,c-\frac{b}{c}(u\cdot c)=\sum_{i=1}^{n-1}(u^1_i\cdot \frac{b}{c})(u^2_{n-i}\cdot c).$$
So we have in $A$ that
$$(u\cdot b)\,c-(u\cdot\frac{b}{c})\,c^2-b(u\cdot c)=\sum_{i=1}^{n-1}(u^1_i\cdot \frac{b}{c})\,c\,(u^2_{n-i}\cdot c).$$
Now, by assumption, $(u\cdot\frac{b}{c})\,c\in Ab$ and $(u^1_i\cdot \frac{b}{c})\,c\in Ab$, so $b$ divides $(u\cdot b)\,c$. And therefore $b$ divides $u\cdot b$. Furthermore we obtain, using induction as in (i), that $c$ divides $b(u\cdot c)$ and therefore that $c$ divides $u\cdot c$.\\
(iii).\ Let $u\in\Dist_n(H)$. By the corollary to Lemma~\ref{lem.distribution} we can write
$$\Delta_m(u)=u\ot\ve\ot\cdots\ot\ve+\cdots+\ve\ot\cdots\ot\ve\ot u+\sum_ju^1_j\ot\cdots\ot u^m_j,$$ where $u^i_j\in\Dist_{n_{ij}}(H)$ for  $n_{ij}\in\{0,\ldots,n-1\}$ with $\sum_{i=1}^mn_{ij}=n$ for all $j$. Then
$$u\cdot(a^m)=ma^{m-1}u\cdot a+\sum_ju^1_j\cdot a\cdots u^m_j\cdot a.$$
So, by induction on $n$, we get that $a$ divides $u\cdot a$ for all $u\in\Dist_n(H)$.
\end{proof}

\begin{cornn}
Assume that $A$ is an $H$-algebra and that $k$ is perfect or $H$ is finite.
\begin{enumerate}[{\rm(i)}]
\item If $H$ is irreducible, then, in the conclusions in Proposition~\ref{prop.equifactorisation}, ``$\Dist(H)$-stable" may be replaced by ``$H$-stable".
\item Assume that $A$ is normal and that every irreducible component of $H$ contains a $k$-point. Then Proposition~\ref{prop.equifactorisation}(iii) is also valid with ``$\Dist(H)$-stable" replaced by ``$H$-stable". Furthermore, if in (ii) we require $A(b/c)$ also to be $H(k)$-stable, then $Ab$ and $Ac$ are $H$-stable.
\end{enumerate}
\end{cornn}

\begin{proof}
(i).\ This follows immediately from \cite[I.7.17(6), 8.6]{Jan}.\\
(ii).\ Let $h\in H(k)$. First consider Proposition~\ref{prop.equifactorisation}(ii). Then, by assumption, $(h\cdot b)c\in Ab(h\cdot c)$ and $(h^{-1}\cdot b)c\in Ab(h^{-1}\cdot c)$. So $b$ divides $h\cdot b$ and $c$ divides $h\cdot c$. Now the result follows from Lemma~\ref{lem.submodule}. Now consider Proposition~\ref{prop.equifactorisation}(iii). Let $h\in H(k)$. By assumption we have $a^m|(h\cdot a)^m$. Then $(h\cdot a)/a\in{\rm Frac}(A)$ is integral over $A$. So $a|h\cdot a$, since $A$ is normal. Now the result follows again from Lemma~\ref{lem.submodule}.
\end{proof}

\begin{remgl}
The arguments in the proof of Proposition~\ref{prop.equifactorisation}(iii) also yield the statement with the property ``$Aa$ is $\Dist(H)$-stable" replaced by ``$a$ is $\Dist(H)$-semi-invariant" or by ``$a$ is $\Dist(H)$-invariant".
\end{remgl}

The theorem below is a generalisation of a well-known result for connected algebraic groups over an algebraically closed field (see e.g. \cite[Lem.~20.1]{Gr}). We remind the reader that the unique factorisation property makes sense for any commutative monoid in which the cancellation law holds. In fact such a monoid has the unique factorisation property if and only if the quotient by the units is a free commutative monoid.

\begin{thmgl}\label{thm.UFD1}
Let $H$ be an irreducible algebraic group scheme over $k$ and let $A$ be a $\Dist(H)$-algebra. Assume that $A$ is a UFD. Then the monoid of nonzero elements $a\in A$ such that $Aa$ is $\Dist(H)$-stable has the unique factorisation property. If ${\rm char}\,k=0$, then its irreducible elements are the irreducible elements $a$ of $A$ such that $Aa$ is $\Lie(H)$-stable. If ${\rm char}\,k=p>0$, then its irreducible elements are the elements $a^{p^s}$, where $a$ is irreducible in $A$, $s\ge0$, $Aa^{p^s}$ is $\Dist(H)$-stable, and $s$ is minimal with this property.
\end{thmgl}

\begin{proof}
The case that ${\rm char}\,k=0$ follows easily from Cartier's Theorem and Proposition~\ref{prop.equifactorisation}, so we assume that ${\rm char}\,k=p>0$. Then we note that the statement of the theorem is equivalent to the following statement. If $a$ is an irreducible factor which occurs to the power $t$ in the prime factorization in $A$ of an element $b\in A$ such that $Ab$ is $\Dist(H)$-stable, then $Aa^{p^s}$ is $\Dist(H)$-stable for some $s\ge0$ and if $s$ is minimal with this property, then $p^s|t$.

So let $a,b,t$ be as stated above. Then, by Proposition~\ref{prop.equifactorisation}(i), $Aa^t$ is $\Dist(H)$-stable. Now write $t=p^rt_1$ with $p\nmid t_1$. Then $a^t=(a^{p^r})^{t_1}$, so $Aa^{p^r}$ is $\Dist(H)$-stable by Proposition~\ref{prop.equifactorisation}(iii). For $s$ as in the theorem we must have $s\le r$ and therefore $p^s|t$.
\end{proof}

To formulate the next theorem correctly we need some notation. Let $H$ be an algebraic group scheme over $k$ and let $A$ be an $H$-algebra which is a UFD. Let $M$ be the quotient of the monoid of nonzero elements in $A$ such that $Aa$ is $H^0$-stable by the units. So $M$ can also be considered as the monoid of nonzero $H^0$-stable principal ideals in $A$. Furthermore, we put $\Gamma=H(k)/H^0(k)$. Note that $\Gamma$ is a finite group and that it acts on $M$.

\begin{thmgl}\label{thm.UFD2}
Let $H$ be an algebraic group scheme over $k$ and let $A$ be an $H$-algebra. Assume that $A$ is a UFD, that every irreducible component of $H$ contains a $k$-point and that $k$ is perfect or $H$ is finite. Let $M$ and $\Gamma$ be as above. Then the monoid of nonzero elements $a\in A$ such that $Aa$ is $H$-stable has the unique factorisation property. Its irreducible elements are the products of representants of the elements in the $\Gamma$-orbits of the irreducible elements of $M$.
\end{thmgl}

\begin{proof}
By Theorem~\ref{thm.UFD1} and Lemma~\ref{lem.submodule}(i) $M$ is a free commutative monoid. By Lemma~\ref{lem.submodule}(ii), an element $Aa$ of $M$ is $H$-stable if and only if it is $\Gamma$-fixed. So the quotient of the monoid from the theorem by the units is $M^\Gamma$. But this monoid is free commutative with its irreducible elements as described in the theorem.
\end{proof}

\begin{remsgl}\label{rems.UFD}
1.\ If $A$ has a $\Dist(H)$-stable filtration $A_0\subseteq A_1\subseteq A_2\cdots$ with $A_0=k$ such that $\gr A$ is a domain, then $Aa$ is $\Dist(H)$-stable if and only if $a$ is a $\Dist(H)$-semi-invariant. This follows from a simple degree comparison. By Lemma~\ref{lem.submodule} we can draw these conclusions for an $H$-action, if we assume that $k$ is perfect or $H$ is finite.\\
2.\ Assume that $k$ is perfect with ${\rm char}\,k=p>0$. Let $H$ be an algebraic group scheme over $k$ and let $A$ be a reduced commutative $H$-algebra. Let $A^{(r)}$ and $H^{(r)}$ be the $r$-th Frobenius twists of $A$ and $H$, see \cite[I.9.2]{Jan}. Then $H^{(r)}$ acts on $A^{(r)}$ and for the isomorphism $a\mapsto a^{p^r}:A^{(r)}\to A^{p^r}$ we have ${\rm Fr}_r(h)\cdot a=h\cdot a^{p^r}$ for all $a\in A$, all $h\in H(R)$ and all commutative $k$-algebras $R$. Here ${\rm Fr}_r:H\to H^{(r)}$ is the $r$-th Frobenius morphism.
By \cite[I.9.5]{Jan} we have $k[H/H_r]=k[H]^{H_r}=k[H]^{p^r}$. So, if for the nil radical $\n$ of $k[H]$ we have $\n^{p^r}=0$, then ${\rm Fr}_r$ induces an isomorphism $H/H_r\stackrel{\sim}{\to}H_{\rm red}^{(r)}$, where $H_{\rm red}$ is the closed subgroup scheme of $H$ defined by $\n$.
So $a\in A$ is an \hbox{$H_{\rm red}$-(semi-)}invariant if and only if it is an $H_{\rm red}^{(r)}$-(semi-)invariant as an element of $A^{(r)}$ if and only if $a^{p^r}$ is an $H$-(semi-)invariant. Similarly, we get that $Aa$ is $H_{\rm red}$-stable if and only if $A^{p^r}a^{p^r}$ is $H$-stable. If $A$ is a normal domain, then this is also equivalent to $Aa^{p^r}$ is $H$-stable.%
\\
3.\ Let $k$ and $H$ be as in the previous remark and assume that $H$ is reduced. Let $A$ be a reduced commutative $\Dist(H)$-algebra. Then, by \cite[I.9.5]{Jan}, ${\rm Fr}_r$ induces an isomorphism $H_{r+s}/H_r\stackrel{\sim}{\to}H^{(r)}_s$. So, as in the previous remark, we obtain that $a\in A$ is an $H_s$-(semi-)invariant if and only if $a^{p^r}$ is an \hbox{$H_{r+s}$-(semi-)}invariant. So the minimal $s$ such that $a^{p^s}$ is an $H_r$-(semi-)invariant is $r-t$ where $t\in\{0,\ldots,r\}$ is maximal with the property that $a$ is an \hbox{$H_t$-(semi-)}invariant. The analogues of the statements from the previous remark about the property that $Aa$ is $H$-stable are also valid.\\
4.\ Under the assumptions of Theorem~\ref{thm.UFD1} we also have that the group of nonzero $\Dist(H)$-stable principal fractional ideals is free abelian. In the case of Theorem~\ref{thm.UFD1} the group of nonzero principal fractional ideals that are $\Dist(H)$-stable and $H(k)$-stable is free abelian. The extra argument one needs is Proposition~\ref{prop.equifactorisation}(ii) and assertion (ii) of its corollary.\\
5.\ The UFD property is badly behaved with respect to field change, see \cite[Chap. 7 Ex. 4 and 6 to \S3]{Bou}.
\end{remsgl}

By Remark~\ref{rems.UFD}.1 filtrations give us an obvious way to deduce that $a$ is an $H$-semi-invariant from the fact that $Aa$ is $H$-stable. It applies, for example, when $A=k[V]$, where $V$ is a finite dimensional $H$-module. In Proposition~\ref{prop.semi-invariant} below we replace the filtration by the action of a ``big" ($A^G=k$) reductive group $G$. It applies, for example, to the case $A=k[G]$ and $H\le G$ acting via the right or left regular action. It also applies to the case $A=k[G]$ and $H\le G\le G\times G$ where $G$ is embedded diagonally and $G\times G$ acts in the usual way (so $G\times G$ is now the big reductive group and $H$ acts via the adjoint action). In both cases we take $H'=1$.

We emphasize that the property that every $a\in A$ such that $Aa$ is $H$-stable is an $H$-semi-invariant is a very general one. For example, when $H$ is a (reduced) closed subgroup of a connected algebraic group acting on an affine variety $X$, then the algebra $k[X]$ has this property. This follows from the result stated at the beginning of Section~\ref{s.intro} by passing to the normalisation of $X$, since, when $X$ is normal, $k[X]f$ is $H$-stable if and only if $(f)$ is $H$-fixed. So Remark~\ref{rems.UFD}.1 and Proposition~\ref{prop.semi-invariant} below are not the definitive results in this direction. The group of order two acting on the one-dimensional torus by inversion is an example where $k[X]$ does not have this property. 
To prove Proposition~\ref{prop.semi-invariant} we need a lemma.

\begin{lemgl}\label{lem.torus}
Assume that $k$ is algebraically closed. Let $T$ be a torus acting on a commutative domain $A$ and let $a,b,c\in A\sm\{0\}$. Assume that $a=bc$ and that for each weight $\chi$ of $T$ the $\chi$-component of $c$ is nonzero if the $\chi$-component of $a$ is nonzero. 
Then $b$ is $T$-fixed.
\end{lemgl}

\begin{proof}
The assumption on the weight components is inherited by any subtorus of $T$, so we may assume that $T$ is one-dimensional. Then the $T$-action amounts to a $\mb Z$-grading of the algebra $A$ and the result follows by comparing highest and lowest degree.
\end{proof}

\begin{propgl}\label{prop.semi-invariant}
Assume that $k$ is algebraically closed. Let $G$ and $H'$ be algebraic group schemes over $k$ acting on a $k$-algebra $A$ and assume that these actions commute. Assume furthermore that $G$ is a connected reductive algebraic group and that $A$ is a commutative domain with $A^G=k$. Let $H$ be a closed subgroup scheme of $G$ and let $a\in A$. If $Aa$ is stable under $H$ and $H'$, then $a$ is a semi-invariant for $H$ and $H'$.
\end{propgl}

\begin{proof}
Let $A_0\subseteq A_1\subseteq A_2\subseteq\cdots$ be the filtration of $A$ from \cite[\S15]{Gr} (first introduced in \cite[\S4]{Pop}). Then this filtration is $G$ and $H'$-stable, $\gr(A)$ is a domain and $A_0$ is spanned by the $G$-semi-invariants. Taking degrees we get that for $h\in H(k)$ and $u\in\Dist(H)$ we have $h\cdot a\in A_0a$ and $u\cdot a\in A_0a$, and the same with $H$ replaced by $H'$. Since characters of $G$ are determined by their restriction to the connected centre, the result follows from Lemma's \ref{lem.torus} and \ref{lem.submodule} and the fact that $A^G=k$.
\end{proof}
We note that, although Proposition~\ref{prop.semi-invariant} is only stated for $k$ algebraically closed, one can in certain situations apply it to the case that $k$ is not algebraically closed by applying field extension from $k$ to $\ov k$. The point is that the properties ``$Aa$ is $H$-stable" and ``$a$ is an $H$-semi-invariant" are well-behaved with respect to field extension and restriction. To make this work one needs, of course, that $\ov k\ot A$ is a domain.

The following result is an immediate consequence of Theorem~\ref{thm.UFD2}, Proposition~\ref{prop.semi-invariant} and the well-known fact that, for $G$ connected reductive, $k[G]$ is a UFD if its derived group is simply connected (see \cite{T} for references and for an elementary proof).

\begin{corgl}
Assume that $k$ is algebraically closed. Let $G$ be a connected reductive algebraic group over $k$ with simply connected derived group and let $H$ be a closed subgroup scheme of $G$. Then the monoid of nonzero $H$-semi-invariants in $k[G]$ under the right regular action has the unique factorisation property. Its irreducible elements are as given by Theorem~\ref{thm.UFD2} with the property ``$Aa$ is $H$-stable" replaced by ``$a$ is an $H$-semi-invariant".
\end{corgl}

We finish the paper with some generalisations of results in \cite{T} and \cite{T2}. For results on generators of algebras of infinitesimal invariant we refer to \cite{D}. We need some terminology.
Assume that $k$ is algebraically closed and let $G$ be a connected reductive group over $k$. As in \cite{KW} we call an element $\xi$ of $\g^*$ {\it semi-simple} if there exists a maximal torus $T$ of $G$ such that $\xi$ vanishes on all the root spaces of $\g$ relative to $T$.

\begin{corgl}
Assume that $k$ is perfect of characteristic $p>0$. Let $G$ be an algebraic group scheme over $k$ such that $G_{\ov k}$ is a connected reductive algebraic group and let $\g$ be its Lie algebra. In the following cases $A^{G_r}$ is a UFD and its irreducible elements are the elements $a^{p^{r-t}}$, where $a$ is irreducible in $A$ and $t\in\{0,\ldots,r\}$ is maximal with the property that $a$ is an $H_t$-invariant.
\begin{enumerate}[{\rm(1)}]
\item If $A=k[G]$ is a UFD,
\item If $A=k[\g]$ and the semi-simple elements are dense in $\Lie(G_{\ov k})=\ov k\ot\g$,
\item If $A=S(\g)=k[\g^*]$ and the semi-simple elements are dense in $(\ov k\ot\g)^*$,
\end{enumerate}
where in each case $G_r\le G$ acts via the adjoint action.
\end{corgl}

\begin{proof}
By Theorem~\ref{thm.UFD1} and Remark~\ref{rems.UFD}.3 we only have to check that in each case $a\in A$ is a $G_r$-invariant whenever $Aa$ is $\Dist(G_r)$-stable. For this we may assume that $k$ is algebraically closed. By Proposition~\ref{prop.semi-invariant} and Remark~\ref{rems.UFD}.1 it remains to show that every semi-invariant of $G_r$ in $A$ is an invariant. Assume, in case (1), that $f\in k[G]$ is a nonzero semi-invariant for $G_r$. By the density of the semi-simple elements in $G$ there exists a maximal torus $T$ such that $f|_T\ne0$. Now we obtain, using the triviality of the adjoint action of $T$ on $k[T]$ as in \cite[Lem.~2]{T}, that $f$ is a $\Dist(T)$-invariant. On the other hand, $f$ is also an invariant of the $\Dist(U_{\alpha,r})$, since the infinitesimal root subgroups are unipotent. So, e.g. by \cite[II.3.2]{Jan}, $f$ is a $\Dist(G_r)$-invariant, that is, a $G_r$-invariant. In the other two cases the proof is completely analogous.
\end{proof}

For the next corollary we need to assume the so-called ``standard hypotheses" for a reductive group in positive characteristic. They can be found in \cite{Jan2} and are also stated in the introduction of \cite{T2}.

\begin{corgl}
Assume that $k$ is algebraically closed of characteristic $p>0$. Let $G$ be a connected reductive algebraic group over $k$, let $\g$ be its Lie algebra, let $U$ be the universal enveloping algebra of $\g$ and let $r\ge1$. Assume that $G$ satisfies the standard hypotheses (H1)-(H3) from \cite{Jan2}. Then $U^{G_r}$ is a UFD.
\end{corgl}

\begin{proof}
We have $U^{G_r}=Z^{G_r}$, where $Z=U^{G_1}=U^\g$ is the centre of $U$. By \cite[Thm.~2]{T2} $Z$ is a UFD and by \cite[1.4]{T2} $Z$ has a $G$-stable filtration $A_0\subseteq A_1\subseteq A_2\subseteq\cdots$ with $A_0=k$ such that $\gr(Z)$ is a domain. So, by Remark~\ref{rems.UFD}.1, we have for $a\in Z$ that $Za$ is $G_r$-stable if and only if $a$ is a $G_r$-semi-invariant. Furthermore, $Z\cong S(\g)^\g\subseteq S(\g)$ as $G$-modules by \cite[1.4]{T2}, so by (iii) of the preceding corollary, every $G_r$-semi-invariant of $Z$ is a $G_r$-invariant. Now the result follows from Theorem~\ref{thm.UFD1}.
\end{proof}

\begin{remsgl}
1.\ If $k$ is algebraically closed, or more generally $G$ is split over $k$, then $k[G]$ is a UFD if and only if the derived group $DG$ is simply connected. The arguments in \cite[Sect.~1]{T} are also valid in the more general case that $G$ is split over $k$.\\
2.\ If $k=\mb R$ and $G=\SO_2$ is the compact real form of the one dimensional complex torus (see \cite[II.8.16]{Bo}), then $\mb R[G]\cong\mb R[a,b]/(a^2+b^2-1)$ is not a UFD. See also \cite[Chap. 7 Ex. 4 to \S3]{Bou}. The same is true for $\SO_2$ over any field in which $-1$ is not a square.\\
3.\ The following example shows that the assumption on the density of the semi-simple elements in case (3) of Corollary~2 cannot be omitted. Let $k$ be algebraically closed of characteristic $2$. Then $\pgl_2^*\cong\spl_2$ as $\GL_2$-modules. So the semi-simple elements are not dense in $\pgl_2^*$. Put $G=\PGL_2$ and let $(h,x,y)$ be a basis of $\pgl_2$ with $[h,x]=x$, $[h,y]=y$ and $[x,y]=0$. Then it follows from the identity $(xy)^p=x^py^p$ that $S(\g)^{G_1}=S(\pgl_2)^{\pgl_2}=k[x^p,y^p,h^p,xy]$ is not a UFD. Note that $x$ and $y$ are $\pgl_2$-semi-invariants that are not $\pgl_2$-invariants.
\end{remsgl}

\noindent{\it Acknowledgement}. This research was funded by a research grant from The Leverhulme Trust.

\bigskip

{\sc\noindent Department of Mathematics,
University of York, Heslington, York, UK, YO10~5DD.
{\it E-mail address : }{\tt rht502@york.ac.uk}
}

\end{document}